\documentclass{llncs}

\usepackage{booktabs} 
\usepackage{graphicx}
\usepackage{cellspace}
\usepackage[english]{babel}
\usepackage{latexsym,layout,verbatim,amssymb,amsmath,setspace,cite}
\usepackage[dvipsnames]{xcolor}
\usepackage{cleveref,enumerate,fontenc}
\usepackage[utf8]{inputenc}
\usepackage[T1]{fontenc}
\newtheorem{ass}{Assumption}

\def\As{\mathcal{A}}
\def\Fs{\mathcal{F}}

\def\erre{\mathbb{R}}
\def\enne{\mathbb{N}}

\def\vp{\varepsilon}

\def\Ec{\mathbb{E}}

\begin{document}

\frontmatter          
\pagestyle{empty}  

\title{A  Girsanov Result through Birkhoff Integral}
\titlerunning{A  Girsanov result ...}  
%
\author{Domenico Candeloro 
\and Anna Rita Sambucini\footnote{corresponding author\\
The  authors have been supported  by Fondo Ricerca di Base 2015 University of Perugia - titles: "$L^p$ Spaces in Banach Lattices with applications", "The Choquet integral with respect to fuzzy measures and applications"
  and by Grant  Prot. N.  U UFMBAZ2017/0000326 of GNAMPA -- INDAM (Italy).}  
}
\authorrunning{Candeloro -- Sambucini} 
\tocauthor{Domenico Candeloro and Anna Rita Sambucini}
\institute{Department of Mathematics and Computer Science, University of Perugia, Italy,\\
\email{ [domenico.candeloro, anna.sambucini]@unipg.it}\\
 orcid ID:{ 0000-0003-0526-5334,
 0000-0003-0161-8729}
}

\maketitle              

\begin{abstract}
A vector-valued version of the Girsanov theorem is presented, for a scalar process with respect to a Banach-valued measure.
Previously, a short discussion about the Birkhoff-type integration is outlined, as for example integration by substitution, in order to fix the measure-theoretic tools needed for the main result, Theorem 6, where a martingale equivalent to the underlying vector probability has been obtained in order to represent the modified process as a martingale  with the same marginals as the original one.
\end{abstract}
\keywords{Girsanov Theorem, martingale, Birkhoff integral}\\
{\bf  2010 AMS  Classification}: {28B05, 60G44}\\
\section{Introduction}\label{one}
In probability theory, the so-called Girsanov Theorem is a well-known result, whose interest lies both
in its theoretical features and in its technical consequences, see for examle \cite{Mik}.
The original formulation of this theorem is related to the Wiener measure, i.e. the {\em distribution} of the standard  Brownian Motion,
$(B_t)_{t\in [0,+\infty[}$, as a stochastic process on a probability space $(\Omega,\mathcal{A},P)$. \\
The
Girsanov Theorem is a fundamental tool for Stochastic Calculus and Randow Walks; this last mathematical model 
has many uses as a simulation tool:
Brownian Motion of Molecules,
stock prices and behavior of investors,
modeling of cascades of neuron firings in brain
and it has important practical uses in the internet:
Twitter uses random walks to suggest who to follow,
Google uses random walks to order pages which
match a search phrase.\\
In many concrete situations, particularly in stochastic calculus, the resultant processes $(\widetilde{B}_t)_t$ are usually obtained as 
suitable {\em transformations} of $(B_t)_t$, and so their distribution is different from the Wiener measure.
A typical situation is the following: assume that $a(t,\omega)$ is a stochastic process adapted to the Brownian Motion $(B_t)_t$, and define:
 $$(\widetilde{B}_t)_t:=\int_0^t a(t,\omega)dt+B(t).$$
Though the distribution of this process is different by the Wiener measure, the
 Girsanov theorem states that it is possible to endow the basic probability space 
with a new probability measure, $Q$ (which turns out to be absolutely continuous w.r.t. $P$), in such a way
that the distribution of  $(\widetilde{B}_t)_t$ under the new probability $Q$ is the Wiener measure, i.e. the same as $(B_t)_t$ under the 
original probability $P$.
This clearly simplifies  all calculations involving just the distribution of $(\widetilde{B}_t)_t$, since in the new probability space this
 process is the same as $(B_t)_t$. 

In the example outlined above, the measure $Q$ can be described by its derivative w.r.t. $P$\ :
$$\dfrac{dQ}{dP}(\omega)=\exp\left\{-\int_0^T a(s,\omega)dB_s-\frac{1}{2}\int_0^T a^2(s,\omega)dt\right\}.$$
We also point out that, in this example, the process 
$$\exp\left\{-\int_0^t a(s,\omega)dB_s-\frac{1}{2}\int_0^t a^2(s,\omega)dt\right\}$$
is a martingale (Novikov condition).

Our research in this paper is motivated by the fact that, when the distributions involved are conditioned by some initial information
 (that can be represented as a particular sub-$\sigma$-algebra $\Fs$ of $\mathcal{A}$), then they should be evaluated with respect to
 $(P|\Fs)$, which is a Banach space-valued measure.

So, in our setting here,
continuing the study started in \cite{cmlsxweber}, 
 changing a bit the notations, the basic space is $(T,\mathcal{A},M)$ where $M:\mathcal{A}\to X$ is a
 Banach-valued $\sigma$-additive measure, and $(w_s)_{s\in [0,S]}$ is a scalar valued process, whose distribution admits a 
density $f_s:\erre\to X$ w.r.t. $M$.
(In this setting, we make use of a Birkhoff-type integral in order to integrate a scalar function with respect to $M$, as well as a 
{\em dual} type of integral in order to integrate a vector-valued function w.r.t. a scalar measure).
In the paper, some conditions are assumed on the functions $f_s$ in order to find an {\em equivalent} measure $\widetilde{M}$,
 under which some transform $\widetilde{w}_s$ of the process  turns out to be a martingale, with the same marginal distributions
 as $(w_s)_s$ with respect to $M$.

\section{Preliminaries}\label{two}

Let $T$ be an abstract, non-empty set, $\mathcal{A}$ 
a $\sigma$-algebra of subsets of $
T$, $\mathcal{B}$  the Borel $\sigma$-algebra in the real line
and \mbox{$\mu :\mathcal{A}\rightarrow \mathbb{R}_0^+$} a non-negative countably additive measure.
Let $(X,\|\cdot \|)$ be a Banach space with the origin 
$0$.
We also will consider measures,  taking values in  $X$: in this case measures will be usually denoted  with letters like $m$ or $M$,
 while functions with capital letters, like $F$ or $W$.
\begin{definition}\label{ex2.4} \rm 
A \textit{partition of} $T$ is a finite or countable family of
nonempty sets $P:=\{A_{n}\}_{n\in \mathbb{N}}
\subset \mathcal{A}$ such that $A_{i}\cap A_{j}=\emptyset
,i\neq j$ and $\cup_{n\in 
\mathbb{N}}A_{n}=T$. \\
 If $P$ and $P^{\prime }$ are two partitions of $T$
, then $P^{\prime }$ is said to be \textit{finer than} $P$, denoted\ by $
 P^{\prime } > P$ , if every set of $P^{\prime }$
is included in some set of $P$.
  The \textit{common refinement} of two  partitions $
P$ and $P^{\prime }$ is the partition $P\vee P^{\prime}$.
\end{definition}

We shall make use of the Birkhoff integral, for two different cases. According with the results obtained in \cite{ccgs2015a}, 
and taking into account that we need   measurable scalar functions $f$ and strongly measurable vector functions $F$,
we can adopt the following version:

\begin{definition}\label{ex3.5} \rm 
Given two pairs $(F,\mu)$ and $(f,m)$ with 
  $F:T\rightarrow X$ and $f: T \to \mathbb{R}$, while $\mu$ and $m$ denote two countably additive measures with values
in $\mathbb{R}_0^+$ and $X$ respectively, then
\begin{description}
\item[$\alpha$)]  a strongly measurable vector function $F$ is  {\em $B_1$-integrable}  on 
$T$ w.r.t $\mu$
\item[$\beta$)]  a measurable scalar function $f$ is {\em $B_2$-integrable} on 
$T$ w.r.t $m$
\end{description}
 if  $\exists \,  I\in X$ with the following property:  $\forall \, \varepsilon >0$, 
 $\exists \, P_{\varepsilon }$ partition 
of $T$ 
so that  $\forall \, P=\{A_{n}\}_{n\in \mathbb{N}}$ of $T$, with $P\geq P_{\varepsilon}$ and
  $\forall \, t_{n}\in A_{n},n\in \mathbb{N}$, one has (respectively) 
\begin{eqnarray}\label{maxlim1}
 \overline{\lim}_n \left\|\left(\sum_{i=1}^nF(t_i)\mu(A_i)-I\right) \right\|\leq \vp;
\end{eqnarray}
\begin{eqnarray}\label{maxlim2}
  \overline{\lim}_n \left\|\left(\sum_{i=1}^nf(t_i)m(A_i)-I\right) \right\|\leq \vp.
\end{eqnarray}
The set $I$ is called the \textit{$B_1$ ($B_2$)  integral  of }$F$  ($f$)
\textit{\ on }$T$ \textit{\ with respect to} $\mu $ ($m$) and is denoted by $\int_{T}Fd\mu , (\int_T f dm);$
the corresponding spaces of $B_i$-integrable functions are denoted with $L_{B_1} (\mu, X)$ and $L_{B_2}(m, \mathbb{R}_0^+)$.
\end{definition}

As proved in \cite[Theorem 3.18]{ccgs2015a}, even if $\mu$ is   $\sigma$-finite  then this notion of integrability is 
equivalent to the classic Birkhoff integrability for Banach-valued strongly measurable mappings. 
Moreover in \cite{cdpmsxweber} $B_1$-integrability is named {\em strong Birkhoff integrability}.
One can easily deduce, by means of a Cauchy criterion, that the $B_1$- or $B_2$-integrability on $T$ implies the same in every
 subset $A\in \mathcal{A}$. For an extensive literature on   the Birkhoff or non absolute integrals see for example \cite{cc,bcs2014, bs2012,bms2011, ccgs2015a, cdpms1,  cdpms2, cdpmsxweber, cs2014, cs2015,cr2004, cr2005, cg2014, cgsub1, cg2016, fremlin, vm, pot2006, pot2007, r2005, r2009b}.

\section{Some Properties of the Birkhoff Integrals}\label{three}

We will deduce now some useful formulas for the notions of Birkhoff integral previously introduced.
 These formulas will also give a link between the  $B_1$- and $B_2$-integral.
First, let us mention a stronger result concerning the $B_1$-integrability. 

\begin{theorem}\label{strobi} {\rm (\cite[Theorem 3.14]{cdpmsxweber})}
Let $F \in L_{B_1} (\mu, X)$, then  $\forall \, \vp>0$ there exists a countable partition $P:=\{A_n, n\in \enne\} \subset  \mathcal{A}$, such that
$$\sum_j \left\|F(t_j)\mu(E_j)-\int_{E_j}F d\mu \right\|\leq \vp$$
holds true, for every  partition $P':=\{E_j,j\in \enne\} > P$ and  $\forall \, t_j\in E_j$.
\end{theorem}

\begin{remark}\label{comune} \rm  If $f$ is measurable and $F$ is   $B_1$-integrable then 
the product $t\mapsto f(t)F(t)$ is strongly measurable.
Moreover, thanks to the  measurability of $f$, we can define a countable measurable partition of $T$, $(H_j)_j:=(\{t\in T:j-1\leq |f(t)|< j\})_j$. 
Since
$F$ is $B_1$-integrable, according with Theorem \ref{strobi}, 
for every $\vp>0$ and for each  integer $j$  there exists a measurable countable partition $\{E_j^k, k\in \enne\}$ of $H_j$ such that
$$\sum_r \left\|F(t_j^r)\mu(E_j^r)-\int_{E_j^r} F d\mu \right\|\leq \dfrac{\vp}{ j 2^j}$$
holds true, for every finer partition $\{E_j^r, r\in \enne\}$ and every choice of points $t_j^r\in E_j^r$. Then, we have also
\begin{eqnarray}\label{doppia}
\sum_j\sum_r \left\| F(t_j^r)f(t_j^r)\mu(E_j^r)-f(t_j^r)M(E_j^r) \right\|\leq 2\vp.
\end{eqnarray}
\end{remark}

\begin{theorem}\label{flauno}{\rm (integration by substitution)}
Given $f:T\to \erre$ and  $F \in L_{B_1} (\mu, X)$, 
  the product $t\mapsto f(t)F(t) \in L_{B_1} (\mu, X)$  iff  
$f \in L_{B_2}(M, \mathbb{R})$,
where $M(A):= \int_A F d\mu$
 and 
\begin{eqnarray}\label{prima}
\int_T f(t)F(t)d\mu=\int_T f(t)dM.
\end{eqnarray}

\end{theorem}

Another useful formula comes from probability theory. We just state it in a particular situation.
Given a measurable $f:T\to \erre$  and a countably additive measure $m:\As \to X$, we can set
$m_f(B)=m(f^{-1}(B))$
for every Borel set $B\in \erre$. Of course, $m_f$ is a countably additive measure, called the
 {\em distribution} of $f$ (with respect to $m$).
We have the following result.
\begin{theorem}\label{fladue}
For any measurable function $g:\erre\to \erre$, one has
$$\int_T g(f)dm=\int_{\erre}g(t)dm_f$$
provided that both $B_2$-integrals exist.
\end{theorem}

We shall denote by $\sigma_f$ the sub-$\sigma$-algebra of $\mathcal{A}$ induced by $f:T\to \erre$, i.e. 
the family of all sets of the type $f^{-1}(B)$, $B\in \mathcal{B}$.

\begin{definition}\label{econd}\rm
Let $F \in L_{B_1} (\mu, X)$. Given any sub-$\sigma$-algebra $\mathcal{E}$ 
of $\mathcal{A}$, the {\em conditional expectation} $\mathbb{E}(F|\mathcal{E})$ (if it exists) is a
 strongly $\mathcal{E}$-measurable mapping $Z$, in $ L_{B_1} (\mu, X)$, such that
$$\int_E F d\mu=\int_E Z d\mu$$
for every $E\in \mathcal{E}$.

In case $\mathcal{E}=\sigma_f$, then we write $Z=\Ec(F|\mathcal{E})=\Ec(F|f)$, and in this case
 $Z$ turns out to be a measurable function of $f$,
  say $Z=h(f)$: then
$$\int_{f^{-1}(B)} F d\mu=\int_{f^{-1}(B)} h(f) d\mu$$
for every Borel set $B$.
\end{definition}

The conditional expectation enjoys several properties, easy to deduce, among which linearity 
with respect to $F$, and the so-called {\em tower property}, i.e., whenever 
$\mathcal{E}\subset \mathcal{G}\subset \mathcal{A}$
$$\Ec(F|\mathcal{E})=\Ec(\Ec(F|\mathcal{G})|\mathcal{E}),$$
provided that all the involved quantities exist.
\\

The next theorem states another important property of the conditional expectation.
\begin{theorem}\label{rimpiazzo}
Let us assume that $\Ec(F|\mathcal{E})$ exists. Then, for every $\mathcal{E}$-measurable 
mapping $g:T\to \erre$ it holds:
$\Ec(F(t)g(t)|\mathcal{E})=g(t)\Ec(F|\mathcal{E})$
provided that $F(t)g(t)  \in L_{B_1} (\mu, X)$.
\end{theorem}

\section{Girsanov Theorem}\label{five}

We shall now state an analogous result as the well-known Girsanov Theorem. With this purpose, 
we shall assume that in the space $(T,\mathcal{A})$ a $\sigma$-additive measure 
$M:\mathcal{A}\to X$ is fixed.
\begin{definition} \rm 
A scalar process $(w_s)_s$ is said to be a {\em Martingale} in itself, if for every $s,v\in [0,S], s<v$, it holds 
$\Ec(w_v|\mathcal{E}_s)=w_s$, i.e. $$\int_E w_v dM=\int_E w_s dM$$
holds true, $\forall \, v,s\in [0,S], \ s<v$, 
and  $\forall \, E\in \mathcal{E}_s$, where $\mathcal{E}_s$ is the least $\sigma$-algebra contained in
 $\mathcal{A}$ such that all $w_r$, $r\leq s$, are measurable.
\end{definition}
\begin{ass}\label{da1a4}
 \rm
 Let us  assume  that a scalar-valued process $(w_s)_{s\in [0,S]}$ is defined, in the space 
$(T,\mathcal{A},M)$, with the property that
\begin{description}
\item[\bf \ref{da1a4}.a)]  $w_s \in L_{B_2}(M, \mathbb{R})$  for each $s$, with null integral,
 and that its distribution $M_s:=M(w_s^{-1}(B))$,  $\forall \, B \in \mathcal{B}$
has a density $f_s \in L_{B_1}(\lambda, X)$

\item[\bf \ref{da1a4}.b)]  let  $\widetilde{w}_s=w_s+sq$, with $q \in \mathbb{R}^+$;
 $\forall \, s$ there exists a  measurable mapping $g_s: T \to \mathbb{R}$ such that
$f_s(x)=g_s(x)f_s(x-qs),$
so that   $\forall \, B \in \mathcal{B}$ 
\begin{eqnarray}\label{servedopo}
M_{s}(B)=\int_B g_s(x)dM_{\widetilde{w}_s};\end{eqnarray}
(We observe that, since $g_s(x)f_s(x-qs) =f_s(x)$ is in $L_{B_1}(\lambda,X)$, from Theorem \ref{flauno} it follows that $g_s$ is $B_2$-integrable w.r.t. $M_{\widetilde{w}_s}$)
\item[\bf \ref{da1a4}.c)]
$\{g_s(\widetilde{w}_s)\}_s $ is a Martingale. 
\end{description}
\end{ass}

As a consequence, we have
\begin{theorem}\label{uno} 
Set for every $A\in \mathcal{A}$,
$Q(A):=\int_A g_S(\widetilde{w_S})dM$.
Under Assumptions \ref{da1a4} it turns out that
$Q_{\widetilde{w}_s}=M_{w_s},$
for every $s\in [0,S]$.
\end{theorem}
\noindent
The previous theorem shows that, under the new measure $Q$, every random variable $\widetilde{w_s} = w_s + sq$ has the same distribution as the corresponding $w_s$ under $M$.

Our next step is to prove that the process 
$\{\widetilde{w_s}\}_s$ is a martingale, under $Q$.
(This property is usually formulated by saying that $Q$ is a {\em Martingale equivalent} measure). To this aim,  we shall assume also the following:
\begin{ass}\label{da5a6}
 \rm
The scalar process $\{ \tilde{w}_t \  g_t (\tilde{w}_t)\}_t$ is a martingale w.r.t. $M$.
\end{ass}
Concerning the last assumption, we remark that, in case $\{w_t\}_t$ is the classical (scalar) Brownian Motion, then the process $\{ \tilde{w}_t g_t(\tilde{w}_t) \}_t$
reduces to $\{ w_t e^{-qw_t-\frac{1}{2}q^2t}  + qte^{-qw_t-\frac{1}{2}q^2t}\}_t$
which shows that the classic Brownian Motion satisfies the Assumption \ref{da5a6}.
So we have
\begin {theorem}\label{due}
 Under Assumptions \ref{da1a4},  \ref{da5a6}
the process $(\widetilde{w_s})_s$ is a martingale with respect to $Q$.
\end{theorem}
\begin{proof}
Fix arbitrarily $s$ and $v$, with $s<v$, and fix $E\in \mathcal{E}_s$.
We observe that
$$\int_E \tilde{w}_vdQ=
\int_Ew_vg_v(\tilde{w}_v)dM+qvQ(E).$$
Since $E\in \mathcal{E}_s$, it is clear that 
$Q(E)=
\int_Eg_s(\tilde{w}_s)dM.$
Therefore,
$$\int_E \tilde{w}_vdQ=\int_E(w_vg_v(\tilde{w}_v)+qvg_v(\tilde{w}_v))dM=\int_E\tilde{w}_vg_v(\tilde{w}_v)dM.$$
By the Assumption \ref{da5a6} it follows then
$$\int_E \tilde{w}_vdQ=
\int_E(w_sg_s(\tilde{w}_s)dM+qsQ(E).$$
But
$$\int_Ew_sg_s(\tilde{w}_s)dM=\int_Ew_sg_S(\tilde{w}_S)dM=\int_Ew_sdQ$$
and in conclusion
$$\int_E \tilde{w}_vdQ=\int_E \tilde{w}_sdQ
,$$
which shows the martingale property.
\end{proof}

\section*{Conclusion}
We have studied some theoretical aspects of the Birkhoff integral, both for scalar valued functions with respect to Banach-valued measures and for the dual situation of vector-valued functions with respect to scalar measures. These previous results are then used in order to state an abstract version of the Girsanov Theorem, where the underlying probability measure  $M$ is Banach-valued. The main results state that, under suitable conditions, a {\em Martingale Equivalent} to $M$ is found, under which the transformed process is a martingale with the same marginals as the original one.


\end{document}